\documentclass[a4paper,reqno, 12pt]{amsart}
\usepackage[T1]{fontenc}
\usepackage[utf8]{inputenc}
\usepackage{amsmath}
\usepackage{amsthm}
\usepackage{amssymb}
\usepackage{mathrsfs}
\usepackage{graphicx}
\usepackage[shortlabels]{enumitem}
\usepackage{mathtools}
\usepackage[usenames,dvipsnames,svgnames,table]{xcolor}
\usepackage[longnamesfirst,numbers,sort&compress]{natbib}
\usepackage{hyperref}
\hypersetup{
colorlinks=true,
linkcolor=black,
urlcolor={blue!60!black},
citecolor=black,
filecolor=black,
pdftitle={Extended formulations for matroid polytopes through randomized protocols},
pdfauthor={Manuel Aprile}} 

\renewcommand{\leq}{\leqslant}

\DeclareMathOperator{\conv}{conv}
\DeclareMathOperator{\xc}{xc}

\newcommand{\R}{\mathbb R}

\newcommand{\calB}{\mathcal B}

\newcommand{\calV}{\mathcal V}

\renewcommand{\conv}{\mathrm{conv}}

\newcommand{\rk}{\mbox{rk}}

\newtheorem{prop}{Proposition}
\newtheorem{conj}{Conjecture}

\newtheorem{lem}[prop]{Lemma}

\newtheorem{thm}[prop]{Theorem}

\newtheoremstyle{question}{}{}{\color{blue}}{}{\color{blue}\bfseries}{}{ }{}
\theoremstyle{question}

\newtheoremstyle{openquestion}{}{}{\color{red}}{}{\color{red}\bfseries}{}{ }{}
\theoremstyle{openquestion}

\title[ ]{Extended formulations for matroid polytopes through randomized protocols}

 \author{Manuel Aprile}

\begin{document}

\maketitle

\makeatletter
\let\old@setaddresses\@setaddresses
\def\@setaddresses{\bigskip\bgroup\parindent 0pt\let\scshape\relax\old@setaddresses\egroup}
\makeatother


\begin{abstract}
 Let $P$ be a polytope.
 The hitting number of $P$ is the smallest size of a hitting set of the facets of $P$, i.e., a subset of vertices of $P$ such that every facet of $P$ has a vertex in the subset. An extended formulation of $P$ is the description of a polyhedron that linearly projects to $P$. We show that, if $P$ is the base polytope of any matroid, then $P$ admits an extended formulation whose size depends linearly on the hitting number of $P$. Our extended formulations generalize those of the spanning tree polytope given by Martin and Wong. Our proof is simple and short, and it goes through the deep connection between extended formulations and communication protocols. 
\end{abstract}

\textbf{Keywords:} extended formulations; matroids; matroid polytope; randomized protocol; spanning tree polytope.

\section{Introduction}

Describing combinatorial problems via geometric objects is a major theme in combinatorial optimization.
For instance, spanning trees in a graph can be described by the spanning tree polytope, which has a well known description \cite{Edmonds71}. However, such polytope has an exponential number of facets, hence its description is too large to be used in practice. In such cases, one can try to add extra variables to the ``natural'' polytope and find an alternative description in an extended space.



 An \emph{extended formulation} of polytope $P\subset \R^d$ is a polyhedron $Q\subset \R^{d+d'}$ that linearly projects to $P$. The \emph{size} of such formulation is its number of inequalities (i.e., the number of facets of $Q$), and the \emph{extension complexity} of $P$, denoted by $\xc(P)$, is the minimum size of an extended formulation of $P$. 
 The systematic study of extended formulations and extension complexity began with Yannakakis \cite{yannakakis1991expressing} and produced a number of impressive results that shed light on the power and on the limits of linear programming \cite{fiorini2015exponential, Rothvoss14}. We refer to \cite{conforti2013extended} for a survey on the topic. 

Let $G$ be an $n$-vertex graph. While the description of the spanning tree polytope of $G$ has $\Omega(2^n)$ inequalities, extended formulations due to Wong \cite{Wong80} and Martin \cite{Martin91} have size $O(n^3)$. Since a cubic number of variables and constraints is still large for practical purposes, a famous open question is whether it is possible to find even smaller extended formulations, see \cite{khoshkhah2018combinatorial, FHJP17, aprile2021smaller}.

Extension complexity is deeply related with the field of communication complexity \cite{kushilevitz2006communication}, which inspired the most celebrated results in the field \cite{fiorini2015exponential, Rothvoss14}. This connection was hinted in \cite{yannakakis1991expressing} and then established in \cite{FFGT15}, where the extension complexity of a polytope $P$ is expressed as the complexity of a randomized protocol solving a certain game on vertices and facets of $P$ (see Section \ref{sec:prot} for details). In particular, \cite{FFGT15} gives a nice, simple protocol for the spanning tree polytope matching the $O(n^3)$ extended formulation from \cite{Martin91}.


Matroids are among the most mysterious objects from the point of view of extension complexity.
The base polytope $B(M)$ of a matroid $M$ is the convex hull of incidence vectors of bases of $M$.  
Bases of general matroids generalize the spanning trees of a graph, hence $B(M)$ is a natural generalization of the spanning tree polytope.  While the optimization problem for matroids is polynomial-time solvable in the oracle model, it is known \cite{rothvoss2013some} that there are matroids $M$ whose extension complexity (by which we mean $\xc(B(M))$) is exponential. However, finding an explicit class of such matroids is a notorious open problem, deeply related to the field of circuit complexity \cite{huynh2016}. 

On the other hand, a number of special classes of matroids have been found to have polynomial extension complexity: graphic and cographic matroids (thanks to the aforementioned formulations of the spanning tree polytope), sparsity matroids \cite{iwata2016extended}, count matroids \cite{conforti2015subgraph}, regular matroids \cite{aprile2019regular}.
In particular, regular matroids are those matroids that can be represented by totally unimodular matrices. This class strictly contains the classes of graphic and cographic matroids and is strictly contained in the class of binary matroids, that can be represented by matrices over the two-elements field GF$(2)$. 
Proving polynomial upper bounds on the extension complexity of binary matroids, or showing a super-polynomial lower bound on it, is a major open question. All the extended formulations proposed so far for special classes of matroids are deeply based on connections with graphs and their structure.


In this paper, we show a general method to derive extended formulations for the base polytope of matroids. This is done through the aforementioned connection between randomized protocols and extended formulations \cite{FFGT15}: in particular we extend the protocol in \cite{FFGT15} for the spanning tree polytope to all matroids, obtaining an extended formulation for $B(M)$ whose size depends on a certain parameter that we introduce.
This parameter can be defined for any polytope as follows: given a polytope $P$, the \emph{hitting number} $h(P)$ of $P$ is the smallest size of a set $\calV$ of vertices of $P$ such that each facet of $P$ contains at least a vertex in $\calV$. 
Our result (Theorem \ref{thm:main}) is that all matroids $M$ on $n$ elements such that $h(B(M))$ is polynomial in $n$ have polynomial extension complexity. The proof of Theorem \ref{thm:main} makes use of the power of randomized protocols (Theorem \ref{thm:random}) and of the bijective basis exchange axiom (Lemma \ref{lem:bijection}); apart from those tools, it is simple and very short.

The hitting set number of a polytope is a natural combinatorial parameter, which, to the best of our knowledge, has not been previously studied (in particular not in relation to extension complexity). 
It is possible that the connection between these two parameters extends to more general polytopes: in fact, we do not know a polytope with exponential extension complexity and polynomial hitting number. 

The paper is structured as follows: in Section \ref{sec:prelim} we give some preliminaries and recall the connection between extension complexity and randomized protocols from \cite{FFGT15}; in Section \ref{sec:proof} we prove our main result; in Section \ref{sec:graphic} we apply our result to the special case of graphic and cographic matroids, and re-derive known bounds on their extension complexity; finally, in Section \ref{sec:conclusion} we suggest possible applications of our result and mention some open questions.


%


\section{Preliminaries}\label{sec:prelim}

We now recall the main notions of matroid theory that we will need. We refer to \cite{oxley2006matroid} for the notions that are not defined.
Throughout the paper, we consider $(M,\calB)$ to be a loopless matroid on ground set $E$ with base set $\calB$.
If $F\subseteq E$ and $e\in E$, we write $F+e$ for $F\cup \{e\}$ and $F-e$ for $F\setminus \{e\}$. A set $F\subseteq E$ is a \emph{flat} if $\rk(F+e)>\rk(F)$ for any $e\in E\setminus F$, where $\rk$ denotes the rank function of $M$.

The \emph{base polytope} of $M$ is $B(M)=\conv\{\chi^B: B\in \calB\}$, where $\chi^B\in \R^E$ is the incidence vector of basis $B$. The following is a description of $B(M)$: 
$$
B(M)=\{x \in \R^{E}_+ \mid x(E)=\rk(E), x(F) \leq \rk(F) \,\forall \text{ flat } F\subseteq E\},
$$
where we write $x(F)=\sum_{e\in F} x_e$ for $F\subseteq E$.

It is easy to see that the description above is redundant, for instance one can restrict to inequalities corresponding to connected flats. The flats of $M$ that induce facets of $B(M)$ are called \emph{flacets} (see \cite{feichtner2005matroid} for a characterization of flacets). The quantity $\rk(F)-|B\cap F|$ is the \emph{slack} of the inequality corresponding to $F$ with respect to vertex $\chi^B$. Basis $B$ is said to have \emph{full intersection} with flat $F$ if it has slack 0 with respect to the corresponding inequality, i.e., if $|B\cap F|=\rk(F)$.

Let $\calV$ be a family of bases of $M$. We say that $\calV$ is a \emph{hitting} family for $M$ if, for any flacet $F$ of $M$, there is at least one basis $B\in \calV$ that has full intersection with $F$. Slightly abusing notation, we denote by $h(M)$ the size of the smallest hitting family for $M$. Notice that this does not correspond exactly to $h(B(M))$, as we are ignoring non-negativity facets, but the two numbers are the same up to an additive $n$ factor, hence the difference is irrelevant to the purpose of this paper.

We mention another, perhaps more famous, polytope related to matroids: the independence polytope $P(M)$ is the convex hull of incidence vectors of independent sets of matroid $M$.
It is well known that $B(M)$ is the face of $P(M)$ defined by $x(E)=\rk(E)$, and that the extension complexities of $P(M)$ and $B(M)$ are the same apart from linear factors: in particular, $\xc(B(M))\leq \xc(P(M))\leq \xc(B(M))+2|E|$. Hence, our Theorem \ref{thm:main} implies similar bounds for independence polytopes as well.

\subsection{Randomized protocols}\label{sec:prot}
We now introduce randomized protocols, referring the reader to \cite{FFGT15} for a more formal definition. Let $P = \conv(\{v_1,\dots, v_N\})=\{x\in \mathbb{R}^d \mid Ax \leq b\}$, where $v_1,\dots, v_N\in\R^d$, $A \in \R^{m \times d}$, $b \in \R^m$. The former description of $P$ is called a \emph{vertical} description, while the latter is called a \emph{horizontal} description.

We consider the following cooperative game between two agents, Alice and Bob: Alice has a row $A_i$ of $A$ as input, Bob has a vertex $v_j$ as input, for some $i\in [m]$, $j\in [N]$, and their goal is to compute the slack $b_i-a_iv_j$ \emph{in expectation}. The agents have unlimited computational power, but do not know each other's input, hence they need to communicate by exchanging bits. Their communication is specified by an algorithm called \emph{randomized protocol}, whose output depends on the agents' inputs, on the communication that takes place between them, and on randomness. The output is then a random variable $X_{i,j}$ depending on the inputs $A_i, v_j$, and the protocol is said to \emph{compute the slack of $P$} (in expectation) if, for any $i\in [m]$, $j\in [N]$, the expectation of $X_{i,j}$ is equal to $b_i-a_iv_j$.
The \emph{complexity} of a randomized protocol is the maximum number of bits exchanged between Alice and Bob on any run of the protocol.

In \cite{FFGT15}  it is shown that, for any polytope $P$, there exists a randomized protocol computing the slack of $P$ of complexity $\lceil \log (\xc(P)) \rceil$. 
Notice that a randomized protocol for $P$ actually depends on the vertical and horizontal representations given: it is customary to use minimal representations, where the $v_j$'s are the vertices of $P$ and the rows of $A$ correspond to facets of $P$ (even though the latter assumption is often relaxed).
 Often, it is more elegant to give a protocol that only considers  the ``non-trivial'' facets of $P$ as possible inputs, for instance ignoring non-negativity inequalities. One can see (Lemma 3 of \cite{FFGT15}) that this still yields an extended formulation for $P$ of approximately the same size. 
 
We report the following upper bound on $\xc(P)$ in the form that we will use in this paper.

\begin{thm}[\cite{FFGT15}]
\label{thm:random}
Let $P=\{x\in [0,1]^d: Ax\leq b\}$ be a polytope such that there is a randomized protocol of complexity $c$ computing the slack of $P$, when Alice's input is restricted to a row of $A$. Then $\xc(P)\leq 2^c+2d$.
\end{thm}

We remark that Theorem \ref{thm:random} only guarantees the existence of the desired formulation; the latter can be written down using Yannakakis' Theorem \cite{yannakakis1991expressing}, but this takes the same time as writing the original description of $P$. In \cite{aprile2020extended} it is shown how to obtain the formulation ``efficiently'' for the special case of deterministic protocols (where no randomness is involved). Although the general case of randomized protocols is open, it is possible to obtain explicit formulations for some randomized protocols by exploiting our knowledge of $P$ (see \cite{aprile2021smaller} for an example).

\section{Proof}\label{sec:proof}
In this section we prove our main result, Theorem \ref{thm:main}. We need the following ``bijective'' basis exchange axiom (see Exercise 12 of Chapter 12 of \cite{oxley2006matroid}, \cite[Corollary 39.12a]{schrijver2003combinatorial}, or \cite{brualdi1969very}).

\begin{lem}\label{lem:bijection}
  Let $B, B'$ be bases of a matroid $M$. Then there is a bijection $\sigma:B\rightarrow B'$ such that, for each $e\in B$, $B-e+\sigma(e)$ is a basis of $M$.
\end{lem}

Given two bases $B,B'$ with $B=\{b_1,\dots,b_r\}$, we say that $B'=\{b_1',\dots, b_r'\}$ is \emph{ordered} with respect to (the given ordering of) $B$ if $b_i'=\sigma(b_i)$ for $i\in [r]$, for a bijection $\sigma$ as in Lemma \ref{lem:bijection}. The following Theorem shows that, given $F\subseteq E$ and a basis $B$ that has full intersection with $F$, we can easily express the slack of inequality $x(F)\leq \rk(F)$ with respect to any other basis $B'$.

\begin{lem}\label{thm:slack}
Let $F\subseteq E$, let $B, B'$ be bases of $M$ with $B=\{b_1,\dots,b_r\}$, and  $B'=\{b_1',\dots, b_r'\}$ ordered with respect to $B$. Assume that $|B\cap F|=\rk(F)$.
Then we have
$$
\rk(F)-|B'\cap F|=|\{i\in [r]: b_i'\not\in F, b_i\in F\}|. 
$$
\end{lem}
\begin{proof}
For $i\in [r]$, there are three cases: 
\begin{enumerate}
    \item $b_i'\in F$. We claim that $b_i\in F$. Indeed, assume by contradiction $b_i\not\in F$: then the intersection of $B-b_i+b_i'$ with $F$ is larger than $|B\cap F|=\rk(F)$, but, since $B'$ is ordered with respect to $B$, $B-b_i+b_i'$ is a basis of $M$, a contradiction.
    \item $b_i'\not\in F$, and $b_i\in F$.
    \item $b_i'\not\in F$, and $b_i\not\in F$.
\end{enumerate}
This determines a partition of $B'$ in $B_1'$, $B_2'$, $B_3'$ according to which case occurs, and also a partition of $B$ into $B_1, B_2, B_3$. Notice that $|B_1'|=|B'\cap F|$, and \[|B'_1|+|B'_2|=|B_1|+|B_2|=|B\cap F|=\rk(F),\]
implying that $\rk(F)-|B'\cap F|=|B_2'|$, which is exactly the thesis.
\end{proof}

 We are now ready to prove our main result.

\begin{thm}\label{thm:main}
For a matroid $M$ of $n$ elements and rank $r$, \[\xc(B(M))= O( h(M)\cdot n \cdot r).\]
\end{thm}
\begin{proof}
We recall the setting of our protocol: Alice has a (facet-inducing) flat $F\subseteq E$, and Bob has a basis of $M$ as input, which we denote by $B'$ for convenience. They both agreed beforehand on a hitting family $\calV$ of size $h(M)$. Alice sends the index of a basis  $B\in \calV$ such that $|B\cap F|=\rk(F)$, using $\log h(M)$ bits. Let $B=\{b_1,\dots, b_r\}$. Bob orders $B'=\{b_1',\dots, b_r'\}$ with respect to $B$.
 He picks an element $i\in[r]$ uniformly at random, and he sends $b'_i$ and the index $i$ to Alice (note that he needs $\log n + \log r$ bits for that). Now, if $b'_i\not\in F$ and $b_i\in F$, then Alice outputs $r$. Otherwise, she outputs 0. The expected output of the protocol is equal to 
$$
\frac{1}{r}\cdot |\{i\in [r]: b_i'\not\in F, b_i\in F\}| \cdot r,
$$
which is equal to the slack of $B'$ and $F$ thanks to Lemma \ref{thm:slack}. The amount of communication needed is $\log h(M)+\log n + \log r$, proving the desired bound thanks to Theorem \ref{thm:random}.
\end{proof}


\section{Graphic and cographic matroids}\label{sec:graphic}

In this section, we derive known results on the extension complexity of the spanning tree polytope as easy consequences of Theorem \ref{thm:main}. 
Let $G(V,E)$ be a connected graph on $n$ vertices, and $P=P(G)$ be its spanning tree polytope. Equivalently, $P$ is the base polytope of the cycle matroid $M=M(G)$ of $G$. The vertices of $P$ are in one-to-one correspondence with the spanning trees of $G$, and a horizontal description of $P$ (see \cite{Edmonds71}) is \[P=\{x\in [0,1]^E: x(E(U))\leq |U|-1, \;\forall \, \emptyset\neq U\subset V, x(E)=|V|-1\},\] where $E(U)$ denotes the set of edges with both endpoints in $U$. 

 In order to apply Theorem \ref{thm:main}, we first restrict to $G=K_n$, 
 the complete graph on $n$ vertices. This corresponds to the most general case, as we will argue later. Notice that $M$ has $\binom{n}{2}$ elements and rank $n-1$. Consider the family $\calV$ containing the star $\delta(v)$ for any vertex $v$ of $K_n$. Notice that any facet-inducing inequality $x(E(U))\leq |U|-1$ has slack 0 with respect to vertex $\chi^T$, with $T=\delta(u)$ and $u\in U$: hence $\calV$ is a hitting family for $M$. Actually, since the inequalities corresponding to a set $U$ with $|U|\leq 1$ are not facet-defining, one can exclude one star from $\calV$ and obtain a hitting family of size $n-1=\rk(M)$, hence $h(M)\leq n-1$.
Blindly applying Theorem \ref{thm:main} would then imply that $\xc(M(K_n))$ is $O(n^4)$. However, the protocol from Theorem \ref{thm:main} can be made cheaper by the following observation.
Recall that, in our setting, Alice has a non-empty subset $U\subset V$ as input, and Bob has a spanning tree $T$ as input. Alice sends a vertex $u$ to indicate the basis $B=\delta(u)\in\calV$, and Bob orders his spanning tree $B'$ according to the bijection $\sigma$ as in Lemma \ref{lem:bijection}. Now, for an edge $e=(v,w)\in B'$, $\sigma(e)$ is either $(u,v)$ or $(u,w)$. Hence, in the protocol given in Theorem \ref{thm:main}, Bob does not need to send the index $i$ along with $e$, but just one bit to indicate one of the two endpoints of $e$. In this way we reconstruct the protocol from \cite{FFGT15} for the spanning tree polytope, where instead of considering the bijection $\sigma$ Bob orients the edges of $T$ away from $u$. Both protocols have complexity $\log \binom{n}{2}+ \log n +1$, which gives the bound $\xc(M(K_n))=O(n^3)$ from \cite{Martin91, Wong80}.

Finally, notice that restricting to complete graphs is without loss of generality, as $P(G)$ for any graph $G$ is obtained as a face of $P(K_n)$, with $n=|V|$. As the extension complexity of a face of $P$ is at most that of $P$, the $O(n^3)$ bound carries over. Moreover, we observe that the above protocol is still valid (and has complexity $\log |V|+ \log |E|+1$) for a general graph $G$ even if the family $\calV$ is not a valid hitting set for $P(G)$: this gives the known bound $\xc(P(G))=O(|V||E|)$. 



Notice that the latter discussion does not imply that $h(M(G))$ is polynomial in $n$ for any $n$-vertex graph $G$: for instance, we do not know a polynomial bound when $G$ is a grid graph.
We suspect this behaviour of function $h$ to hold more in general: even if, for a given matroid $M$, $h(M)$ may be large or just difficult to bound, a strategy to bound $\xc(B(M))$ is to find a matroid $M'$ such that i)  $M$ is a minor of $M'$ with $\rk(M)=\rk(M')=r$; ii) $h(M')$ is polynomial in $r$. In Section \ref{sec:conclusion}, we suggest an application of this principle to regular matroids.

Finally, let $M^*$ be the dual matroid of any matroid $M$. We remark that, since $B(M)$ and $B(M^*)$ are affinely isomorphic, one has $h(M)=h(M^*)$. In particular, it follows that $h(M^*(K_n))\leq n-1$.

\section{Open questions}\label{sec:conclusion}

The main motivation that led us to Theorem \ref{thm:main} was to find an alternative proof for the fact that regular matroids have polynomial extension complexity \cite{aprile2019regular}, as the original proof is quite involved. As most results on regular matroids, the result from \cite{aprile2019regular} is based on Seymour's decomposition \cite{seymour1986decomposition}. This fundamental theorem states that every regular matroid $M$ can be obtained as a 1,2 or 3-sum of graphic, cographic matroids and a special matroid on 12 elements (where 1,2,3-sums are operations between matroids that we do not define here, referring instead to \cite{seymour1986decomposition}). Hence $M$ can be decomposed into smaller building blocks, that have polynomial extension complexity. Showing that the operations of 1,2,3-sums are ``well behaved'' with respect to extension complexity and applying induction is a natural approach in order to bound $\xc(B(M))$, that is pursued in \cite{aprile2019regular}. This is easy for 1-sums and 2-sums, for which one can show that $\xc(B(M))\leq \xc(B(M_1))+\xc(B(M_2))$, where $M$ is the 1-sum or the 2-sum of matroids $M_1$ and $M_2$. However, the situation is much more complex for 3-sums. This leads to a long proof and a bound $\xc(B(M))=O(n^6)$ which the authors of \cite{aprile2019regular} believe to be far from optimal.

In a similar spirit, one can ask whether the operations of 1,2,3-sums are well behaved with respect to the hitting number. Using simple facts on the polyhedral structure of 1,2-sums from \cite{aprile20182}, one can check that $h(M)\leq h(M_1)+h(M_2)$ where $M$ is again the 1-sum or the 2-sum of $M_1$ and $M_2$. But, again, it is not clear whether a similar bound holds for the 3-sum operation. However, one could bypass Seymour's decomposition entirely and study $h(M)$ for a general regular matroid $M$. As mentioned at the end of Section \ref{sec:graphic}, it makes sense to restrict to those matroids $M$ that are ``maximal'' or ``complete'', similarly as we restrict to complete graphs when considering graphic matroids. Recall that a regular matroid $M$ is represented by a totally unimodular matrix $A$, which we can assume to have full row rank equal to $\rk(M)$. 
We say that $M$ (equivalently, $A$) is \emph{complete} if $A$ is maximally totally unimodular, i.e., if adding to $A$ any column that is not a multiple of a column of $A$ violates total unimodularity. 
Any totally unimodular matrix (without repeated columns) is a submatrix of some complete matrix of the same rank. Moreover, the number of columns (elements) we need to add before we obtain a complete matrix (matroid) is quadratic in the rank: a result of Heller \cite{heller1957linear} implies that a regular matroid with rank $r$ has at most $\frac{r(r+1)}2$ elements, apart from loops and parallel elements. This, in particular, shows that graphic matroids of complete graphs are complete.

\begin{conj}\label{conj:regular}
Let $M$ be a complete regular matroid with rank $r$. Then $h(M)$ is bounded by a polynomial in $r$. 
\end{conj}

A proof of Conjecture \ref{conj:regular} would imply, via Theorem \ref{thm:main}, a polynomial bound on the extension complexity of \emph{any} regular matroid $M$. A stronger version of the conjecture might hold, with $h(M)=O(r)$, as is the case for graphic matroids. This would imply that, for any regular matroid $M$ with rank $r$, $\xc(B(M))=O(r^4)$.

Finally, a bound similar to the one given in Theorem \ref{thm:main} could extend to more general polytopes. 
However, a polynomial bound of $\xc(P)$ in terms of $h(P)$ for any polytope $P$ would have very strong consequences: for instance it would imply that stable set polytopes of perfect graphs, and more generally 2-level polytopes, have polynomial extension complexity, which is a notorious open question (see \cite{aprile2018thesis, aprile2017extension, macchia2018two}). Indeed, such polytopes have linear hitting number. A more moderate goal would be to show that $\xc(P)$ is quasipolynomially bounded in terms of $h(P)$ and their dimension $d$: this would generalize the $d^{O(\log(d))}$ bound on stable set polytopes of perfect graphs on $d$ vertices from \cite{yannakakis1991expressing} to all 2-level polytopes, solving a prominent open question related to the log-rank conjecture \cite{lovett2014recent} (see \cite{aprile2018thesis} for details on this connection). We leave further investigations on the subject as an open question. In particular, is there a polytope $P$ such that $\xc(P)$ is exponentially larger than $h(P)$?

\section*{Acknowledgements}
We would like to thank Marco Di Summa and Yuri Faenza for their valuable feedback on the paper, and Samuel Fiorini and Tony Huynh for helpful discussions on the extension complexity of matroid polytopes.
This work was supported by a grant SID 2019 of the University of Padova.

\bibliographystyle{plain}
\bibliography{main}
\end{document}